\documentclass{lms}
\usepackage{amssymb}
\usepackage{enumerate}
\usepackage{amsmath}

\newtheorem{theorem}{Theorem}[section] 
\newtheorem{lemma}[theorem]{Lemma}     
\newtheorem{corollary}[theorem]{Corollary}

\newnumbered{assertion}{Assertion}    
\newnumbered{conjecture}{Conjecture}  
\newnumbered{definition}[theorem]{Definition}
\newnumbered{hypothesis}{Hypothesis}
\newnumbered{remark}[theorem]{Remark}
\newnumbered{note}{Note}
\newnumbered{observation}{Observation}
\newnumbered{problem}{Problem}
\newnumbered{question}{Question}
\newnumbered{algorithm}{Algorithm}
\newnumbered{example}[theorem]{Example}
\newunnumbered{notation}{Notation} 

\title[Semigroup compactifications in terms of filters]
 {Semigroup compactifications in terms of filters} 

\author{T. Alaste}

\classno{54D80, 54D35 (primary)}

\begin{document}
\maketitle

\begin{abstract}
We present a study of semigroup compactifications of a semitopological
semigroup $S$ using certain filters on $S$.
We characterize closed subsemigroups and
closed left, right, and two-sided ideals in any semigroup
compactification of any semitopological semigroup $S$ in terms of
these filters and in terms of ideals of the
corresponding $m$-admissible subalgebra of $C(S)$.
Furthermore,
we characterize those points in any semigroup compactification of $S$ which
belong either to the smallest ideal of the semigroup compactification
or to the closure of this smallest ideal.
\end{abstract}

\section{Introduction}
\label{sec:introduction}

Filters have proven to be an extremely powerful tool while analyzing
algebraic properties of the Stone-\v Cech compactification $\beta S$
of a discrete semigroup $S$.
Indeed, when the algebraic structure of $\beta S$ is considered, it is customary
to consider $\beta S$ as the space of all ultrafilters on $S$ (see
\cite{HindStra}).
In general, any semigroup compactification of a
semitopological semigroup $S$ can be realized as the spectrum of
some $m$-admissible subalgebra of $C(S)$ (see \cite{BJM}).
In this paper, we apply filters to study any semigroup
compactification of a semitopological semigroup $S$.

The Stone-\v Cech compactification of a discrete semigroup is
the most familiar example of a semigroup compactification which may be
considered as a space of filters,
but more general semigroup compactifications have also been studied
using filters.
The $\mathcal{LUC}$-compactification of a topological group $G$
was represented as the space of all near
ultrafilters on $G$ by Ko\c{c}ak and Strauss in \cite{near}.
Near ultrafilters on $G$ need not be filters in the ordinary sense of
the word, since they need not be closed under
finite intersections.
Recently, a representation of the $\mathcal{LUC}$-compactification of
$G$ using filters was given by the author in \cite{ufilters}.
The $\mathcal{LUC}$-compactification of a locally compact
topological group $G$ was also studied using filters
by Budak and Pym in \cite{BudPym},
where the $\mathcal{LUC}$-compactification of $G$ was considered
as a suitable quotient space of the Stone-\v Cech compactification of
$\beta G_d$.
Here, $G_d$ denotes the group $G$ endowed with the discrete topology.
The $\mathcal{WAP}$-compactification of a discrete semigroup
was studied using filters by
Berglund and Hindman in \cite{wapfilt} and 
a treatment of semigroup compactifications using certain equivalence
classes of $z$-filters was given by Tootkaboni in \cite{toot1}.

In this paper, we proceed as follows.
In Section \ref{sec:preliminaries}, we gather some terminology and
notation that we use throughout this paper.
Also, we describe briefly a consideration of a
semigroup compactification of a semitopological semigroup $S$ as a
space of certain filters on $S$.
In Section
\ref{sec:semigr-comp},
we describe the semigroup operation of a semigroup
compactification of $S$ in terms of  these filters.
In Section \ref{sec:closed-subsets-sigma}, we characterize
closed subsemigroups and closed left, right, and two-sided ideals in
any semigroup compactification of $S$ in terms of the introduced
filters and in terms of closed, proper ideals of the corresponding
$m$-admissible subalgebra of $C(S)$.
The last section contains a study of the smallest
ideal of a semigroup compactification of $S$.
We characterize those
points in any semigroup compactification of $S$ which are either in the
smallest ideal of the semigroup compactification in question or in the
closure of this smallest ideal.

\section{Preliminaries}
\label{sec:preliminaries}

We gather some terminology and introduce some notation that we
use throughout this paper.

Let $X$ be a non-empty set.
We denote by $\ell ^{\infty }(X)$ the
$C^{\ast }$-algebra of all bounded, complex-valued functions on $X$.
If $X$ is a topological space,
then we denote by $C(X)$ the $C^{\ast }$-subalgebra of $\ell ^{\infty }(X)$
consisting of continuous members of $\ell ^{\infty }(X)$.
Also, if $X$ is a topological space and $A\subseteq X$, then we denote
by $\text{int}_{X}(A)$ and $\text{cl}_{X}(A)$ the
interior and the closure of $A$ in $X$, respectively.

Let $S$ be a semigroup.
For every element $s\in S$,
the functions $\lambda _s:S\rightarrow S$ and $\rho _s:S\rightarrow S$
are given by $\lambda _s(t)=st$ and $\rho _s(t)=ts$ for every $t\in
S$, respectively.
If $s\in S$ and $A\subseteq S$,
then we frequently denote the sets $\lambda
_s^{-1}(A)$ and $\rho _s^{-1}(A)$ by
$s^{-1}A$ and $As^{-1}$, respectively.
For every function $f:S\rightarrow \mathbb{C}$ and for every element
$s\in S$, the functions $L_sf:S\rightarrow
\mathbb{C}$ and $R_sf:S\rightarrow \mathbb{C}$ are defined by
$L_sf=f\circ \lambda _s$ and $R_sf=f\circ \rho _s$, respectively.

A \emph{right topological semigroup} is a semigroup $S$ endowed with a
Hausdorff topology such that the function $\rho _s$ from
$S$ to $S$ is continuous for every $s\in S$.
The \emph{topological centre} of a right topological semigroup $S$ is
the set
$\Lambda (S)=\{ s\in S : \lambda _s:S\rightarrow S \text{ is
  continuous} \}$.
A \emph{semitopological semigroup} is a right topological semigroup
$S$ such that $\Lambda (S)=S$.

For the rest of this paper, let $S$ be a semitopological semigroup.
We recall briefly some definitions related to semigroup
compactifications.
For more details, see \cite{BJM}.

Let $\mathcal{F}$ be a $C^{\ast }$-subalgebra of $C(S)$ such that
$\mathcal{F}$ contains the constant functions.
By an \emph{ideal} of $\mathcal{F}$, we always mean a closed, proper ideal of
$\mathcal{F}$.
The \emph{spectrum} of $\mathcal{F}$ is the set
\[
\Delta =\{ \mu \in \mathcal{F}^{\ast } : \mu \neq 0 \text{ and } \mu
(fg)=\mu (f)\mu (g) \text{ for all } f,g\in \mathcal{F} \},
\]
where $\mathcal{F}^{\ast }$ denotes the topological dual of
$\mathcal{F}$.
The \emph{evaluation mapping} $e:S\rightarrow \Delta $ is
defined by $[e(s)](f)=f(s)$ for every $s\in S$ and for every
$f\in \mathcal{F}$.
Under the relative weak$^{\ast }$ topology of $\mathcal{F}^{\ast }$,
the space $\Delta $ is a compact Hausdorff space and $\varepsilon (S)$
is dense in $\Delta $.

A non-empty subset $\mathcal{F}$ of $C(S)$ is \emph{left} [\emph{right}]
\emph{translation invariant} if and only if $L_sf\in \mathcal{F}$
[$R_sf\in \mathcal{F}$] for every $f\in \mathcal{F}$ and for every
$s\in S$,
and $\mathcal{F}$ is \emph{translation invariant} if and only if
$\mathcal{F}$ is both left translation invariant and right translation
invariant.
If $\mathcal{F}$ is a translation invariant linear subspace of $C(S)$
and $\mu \in \mathcal{F}^{\ast } $,
then the \emph{left introversion operator} determined by $\mu $ is
defined by $(T_{\mu }f)(s)=\mu (L_sf)$ for every $f\in \mathcal{F}$
and for every $s\in S$.
An \emph{m-admissible subalgebra} of $C(S)$ is a translation
invariant $C^{\ast }$-subalgebra $\mathcal{F}$
of $C(S)$ such that $\mathcal{F}$ contains the constant functions and
$T_{\mu }f\in \mathcal{F}$ for every $f\in \mathcal{F}$ and for every
$\mu \in \Delta $.

A \emph{semigroup compactification} of $S$ is a pair $(\psi ,X)$,
where $X$ is a compact right topological semigroup and $\psi
:S\rightarrow X$ is a continuous homomorphism such that $\psi (S)$ is
dense in $X$ and $\psi (S)\subseteq \Lambda (X)$.
If $\mathcal{F}$ is an $m$-admissible subalgebra of $C(S)$, then
$(e,\Delta )$ is a semigroup compactification of $S$.
The semigroup operation of $\Delta $ is defined by the rule $(\mu \nu
)(f)=\mu (T_{\nu }f)$ for every $f\in \mathcal{F}$.
Conversely,
if $(\psi ,X)$ is a semigroup compactification of $S$,
then the set $\mathcal{F}=\{ h\circ \psi : h\in C(X) \}$ is an
$m$-admissible subalgebra of $C(S)$ and $X$ is isomorphic with $\Delta $.

For the rest of this section, let $\mathcal{F}$ be an $m$-admissible
subalgebra of $C(S)$.
We describe how the spectrum of $\mathcal{F}$
can be viewed as a space of filters on $S$.
Construction of $\Delta $ as the space of all
$\mathcal{F}$-ultrafilters (defined below) on $S$ follows somewhat
similar steps as used in the construction of
the \emph{Smirnov compactification} of the proximity space $(S,\delta
(\mathcal{F}))$ using maximal round filters (see
\cite{Naimpally}), so we omit the details.
Here, $\delta (\mathcal{F})$ is the proximity on $S$ generated by
$\mathcal{F}$.

We denote by $\tau (\mathcal{F})$ the weak topology on $S$ generated
by $\mathcal{F}$.
For every subset $A$ of $S$, we denote by $A^{\circ }$ the $\tau
(\mathcal{F})$-interior of $A$.
An \emph{$\mathcal{F}$-filter} on $S$ is a filter $\varphi $ on $S$ such
that,
for every set $A\in \varphi $ with $A\neq S$, there exist a set
$B\in \varphi $ and a function $f\in \mathcal{F}$ with $f(B)=\{0\}$ and
$f(S\setminus A)=\{1\}$.
An \emph{$\mathcal{F}$-ultrafilter} on $S$ is a maximal (with respect
to inlclusion) $\mathcal{F}$-filter
$\varphi $ on $S$.
We apply the following property of $\mathcal{F}$-filters frequently:
If $\varphi $ is an $\mathcal{F}$-filter on $S$ and $A\subseteq
S$, then $A\in \varphi $ if and only if $A^{\circ }\in \varphi $.

Define $\delta S=\{ p : p \text{ is an $\mathcal{F}$-ultrafilter on }
S\}$.
The \emph{canonical mapping} $\varepsilon :S\rightarrow \delta S$ is
defined as follows.
If $s\in S$, then the family $\mathcal{N}_{\mathcal{F}}(s)$ of all
$\tau (\mathcal{F})$-neighborhoods of $s$ is an
$\mathcal{F}$-ultrafilter on $S$.
Then $\varepsilon (s)=\mathcal{N}_{\mathcal{F}}(s)$ for every $s\in S$.
Define $\widehat{A}=\{ p\in \delta S : A\in p\}$ for every $A\subseteq S$.
The family $\{ \widehat{A} : A\subseteq S \}$ is a base for a topology
on $\delta S$ and,
under this topology, $\delta S$ is homeomorphic
with $\Delta $.
If $f\in \mathcal{F}$, then there exists a unique function
$\widehat{f}\in C(\delta S)$ satisfying $f=\widehat{f}\circ
\varepsilon $.

We denote by $\overline{Y}$ the closure of a subset $Y$ of
$\delta S$ with one exception:
For every $A\subseteq S$, we denote the closure of $\widehat{A}$
by $\text{cl}_{\delta S}(\widehat{A})$.
If $A$ is a $\tau (\mathcal{F})$-open subset of $S$, then the equality
\begin{equation}
  \label{sf:eq:3}
\text{cl}_{\delta S}(\widehat{A})=\overline{\varepsilon (A)}
\end{equation}
holds.
For every $\mathcal{F}$-filter $\varphi $ on $S$,
put
$\widehat{\varphi }=\{ p\in \delta S : \varphi \subseteq p \}$ and
$\overline{\varphi }=\bigcap _{A\in \varphi } \overline{\varepsilon (A)}$.
Then $\widehat{\varphi }=\overline{\varphi }$, and so
$\widehat{\varphi }$ is a non-empty, closed subset of $\delta S$.
Conversely, for every non-empty, closed subset $C$ of $\delta S$,
there exists a unique $\mathcal{F}$-filter $\varphi $ on $S$ such that
$\widehat{\varphi }=C$.

For every $f\in \mathcal{F}$ and for every $r>0$,
put $S(f,r)=\{ s\in S : \lvert f(s)\rvert \leq r \}$.
Let $I$ be an ideal of $\mathcal{F}$ and
put $\mathcal{B}(I)=\{ S(f,r) : f\in I , \, r>0 \}$.
Then $\mathcal{B}(I)$ is a filter base on $S$ and the filter on $S$
generated by $\mathcal{B}(I)$ is an $\mathcal{F}$-filter.
Conversely,
if $\varphi $ is an $\mathcal{F}$-filter on $S$,
then there exists a unique ideal $I$ of $\mathcal{F}$ such that $\varphi
$ is generated by $\mathcal{B}(I)$.
For every $\mu \in \Delta $,
let $p(\mu )$ be the $\mathcal{F}$-ultrafilter on $S$ generated by
$\mathcal{B}(\ker \mu )$.
Then the mapping $\mu \mapsto p(\mu )$ from $\Delta $ to $\delta S$ is
a homeomorphism such that
$e(s)\mapsto \varepsilon (s)$ for every $s\in S$.

\section{Semigroup compactifications}
\label{sec:semigr-comp}

For the rest of this paper, let $\mathcal{F}$ be an $m$-admissible
subalgebra of $C(S)$.
Since the mapping $\mu \mapsto p(\mu )$ from $\Delta $ to $\delta S$
is a homeomorphism and $e(s)\mapsto \varepsilon (s)$ for every $s\in
S$,
the pair $(\varepsilon ,\delta S)$ is a semigroup compactification of $S$.
We proceed to describe the semigroup operation of $\delta S$.

\begin{definition}
\label{sf-Omega_def}
Let $\varphi $ be an $\mathcal{F}$-filter on $S$ and let $A\subseteq
S$.
Define
\[
\Omega _{\varphi }(A)= \{ s\in S : s^{-1}A\in \varphi \}.
\]
\end{definition}

The next lemma follows immediately from the previous definition.

\begin{lemma}
\label{sf-Omega_prop}
Let $\varphi $ and $\psi $ be $\mathcal{F}$-filters on $S$ and
$A,B\subseteq S$.
The following statements hold:
\begin{enumerate}[\upshape (i)]
\item\label{sf-Op_i}
If $A\subseteq B$, then $\Omega _{\varphi }(A)\subseteq
\Omega _{\varphi }(B)$.
\item\label{sf-Op_ii}
If $\varphi \subseteq \psi $, then $\Omega _{\varphi }(A)\subseteq
\Omega _{\psi }(A)$.
\item\label{sf-Op_iii}
$\Omega _{\varphi }(A\cap B)=\Omega _{\varphi }(A)\cap \Omega
_{\varphi }(B)$.
\end{enumerate}
\end{lemma}

In the next lemma,
we need to consider several cases while describing the
semigroup operation of $\delta S$.
Unlike with the Stone-\v Cech compactification of a discrete semigroup,
the subset $\widehat{A}$ of $\delta S$ for some subset $A$
of $S$ need not be closed.
Also, for an arbitrary subset $A$ of $S$,
the sets $\widehat{A}$ and $\varepsilon (A)$ can be very different.
If $S$ is discrete and $\mathcal{F}=\ell ^{\infty }(S)$, then
our results below agree with the known properties of the semigroup
operation of $\beta S$ (see \cite[p. 76]{HindStra}).

\begin{lemma}
\label{sf-product}
Let $A\subseteq S$, let $s\in S$, and let $p,q\in \delta S$.
The following statements hold:
\begin{enumerate}[\upshape (i)]
\item\label{sf-sfp_i}
If $A\in \varepsilon (s)q$, then $s^{-1}A\in q$.
\item\label{sf-sfp_ii}
If $A$ is $\tau (\mathcal{F})$-open and $q\in \text{cl}_{\delta
  S}(\widehat{s^{-1}A})$,
then $\varepsilon (s)q\in \text{cl}_{\delta S}(\widehat{A})$.
\item\label{sf-sfp_iia}
If $q\in \overline{\varepsilon (s^{-1}A)}$, then $\varepsilon (s)q\in \overline{\varepsilon (A)}$.
\item\label{sf-sfp_iii}
If $A\in pq$, then $\Omega _q(A)\in p$.
\item\label{sf-sfp_iv}
If $A$ is $\tau (\mathcal{F})$-open and $p\in \text{cl}_{\delta S} (\widehat{\Omega  _q(A)})$, then
$pq\in \text{cl}_{\delta S}(\widehat{A})$.
\item\label{sf-sfp_iva}
If $p\in \overline{\varepsilon (\Omega  _q(A))}$, then $pq\in \overline{\varepsilon (A)}$.
\end{enumerate}
\end{lemma}

\begin{proof}
(\ref{sf-sfp_i}) Suppose that $A\in \varepsilon (s)q$.
Since $\lambda _{\varepsilon (s)}$ is continuous on $\delta S$, there exists a
$\tau (\mathcal{F})$-open subset $B$ 
of $S$ such that $B\in q$ and $\lambda_{\varepsilon (s)}(\widehat{B})\subseteq
\widehat{A}$.
Then $\varepsilon (sx)\in \widehat{A}$ for every $x\in B$, so
$sx\in A$ for every $x\in B$, and so $B\subseteq s^{-1}A$.
Therefore, $s^{-1}A\in q$.

(\ref{sf-sfp_ii}) Suppose that $A$ is a $\tau (\mathcal{F})$-open subset of
$S$ and that $q\in \text{cl}_{\delta S}(\widehat{s^{-1}A})$.
If $B\in \varepsilon (s)q$,
then $\lambda _s^{-1}(B^{\circ })\in q$,
so $\lambda _s^{-1}(B^{\circ })\cap s^{-1}A\neq
\emptyset $, and so $B^{\circ }\cap A\neq \emptyset $.
Therefore, $\widehat{B}\cap \widehat{A}\neq \emptyset $, as required.

(\ref{sf-sfp_iia}) Suppose that $q\in \overline{\varepsilon (s^{-1}A)}$.
If $B\in \varepsilon (s)q$, then $C:=\lambda _s^{-1}(B^{\circ })\in q$ by statement
(\ref{sf-sfp_i}),
so $\widehat{C}\cap \varepsilon (s^{-1}A)\neq \emptyset $,
so $C\cap s^{-1}A\neq \emptyset $,
and so $B^{\circ }\cap A\neq \emptyset $.
Therefore, $\widehat{B}\cap \varepsilon (A)\neq \emptyset $, as required.

(\ref{sf-sfp_iii}) Suppose that $A\in pq$.
Since $\rho _q$ is continuous on $\delta S$, there exists a $\tau
(\mathcal{F})$-open subset $B$ of $S$
such that $B\in p$ and $\rho _q(\widehat{B})\subseteq
\widehat{A}$.
Then $\varepsilon (s)q\in \widehat{A}$ for every $s\in B$, and so
$B\subseteq \Omega _q(A)$ by statement (i).
Therefore, $\Omega _q(A)\in p$.

(\ref{sf-sfp_iv}) Suppose that $A$ is a $\tau (\mathcal{F})$-open
subset of $S$.
Suppose also that $pq\notin
\text{cl}_{\delta S}(\widehat{A})$.
Then there exists a $\tau (\mathcal{F})$-open subset $B$ of $S$ such
that $B\in pq$ and $\widehat{B}\cap \widehat{A}=\emptyset $,
and so $A\cap B=\emptyset $.
Then $\Omega _q(A)\cap \Omega _q(B)=\emptyset $
by Lemma \ref{sf-Omega_prop} (\ref{sf-Op_iii}),
and so $\widehat{\Omega _q(A)}\cap \widehat{\Omega _q(B)}=\emptyset $.
Since $\Omega _q(B)\in p$ by statement (\ref{sf-sfp_iii}),
we have $p\notin \text{cl}_{\delta S}(\widehat{\Omega _q(A)})$.

(\ref{sf-sfp_iva}) 
Suppose that $pq\notin \overline{\varepsilon (A)}$.
Pick a $\tau (\mathcal{F})$-open subset $B$ of $S$
with $B\in pq$ and $\widehat{B}\cap \varepsilon (A)=\emptyset $.
Then $A\cap B=\emptyset $,
so $\Omega _q(A)\cap \Omega _q(B)=\emptyset $,
and so $\widehat{\Omega _q(B)}\cap \varepsilon (\Omega _q(A))=\emptyset $.
Therefore, $p\notin \overline{\varepsilon (\Omega _q(A))}$ by statement
(\ref{sf-sfp_iii}).
\end{proof}

The next corollary follows from statements (\ref{sf-sfp_iii}) and
(\ref{sf-sfp_iv}) of the previous lemma.

\begin{corollary}
\label{sf-sfp_cor}
Let $p,q\in \delta S$ and let $A\subseteq S$.
The following statements hold:
\begin{enumerate}[\upshape (i)]
\item\label{sf-sfp_cor_i}
If $A\in pq$, then there exist a set $B\in p$
and a family $\{ C_s : s\in B \}$ of members of $q$ such that
$\bigcup _{s\in B}sC_s\subseteq A$.
\item\label{sf-sfp_cor_ii}
If there exist a set $B\in p$ and a
family $\{ C_s : s\in B \}$ of members of $q$ with
$\bigcup _{s\in B}sC_s\subseteq A^{\circ }$,
then $pq\in \text{cl} _{\delta S}(\widehat{A})$.
\end{enumerate}
\end{corollary}

Statement (\ref{sf-sfp_ii}) of the previous lemma does not hold for arbitrary subsets of $S$.
For example, consider the multiplicative semigroup $S=[0,\infty [$
with the Euclidean topology
and let $\mathcal{F} $ be any $m$-admissible subalgebra of $C(S)$.
Put $s=0$ and $A=\{0\}$.
Then $s^{-1}A=S$, and so $\widehat{s^{-1}A}=\delta S$.
Since $A^{\circ }=\emptyset $, we have $\widehat{A}=\emptyset $.
However, if $S$ is algebraically a group and $A$ is any subset of $S$,
then $A\in \varepsilon (s)q$ if and only if $s^{-1}A\in q$.


\section{Closed subsets of $\delta S$}
\label{sec:closed-subsets-sigma}

We proceed to characterize closed
subsemigroups and closed left, right, and two-sided ideals of $\delta
S$ both in terms of $\mathcal{F} $-filters and the corresponding
ideals of $\mathcal{F} $.
For these purposes, we introduce the following lemma.



\begin{lemma}
\label{sf-sec4lemma}
Let $f\in \mathcal{F}$, let $s\in S$, let $p\in \delta S$, and let
$r>0$.
The following statements hold:
\begin{enumerate}
\item \label{sf-sec4lemma_i}
$\widehat{L_sf}=L_{\varepsilon (s)}\widehat{f}$ and
$\widehat{T_pf}=\rho _p\widehat{f}$.
\item \label{sf-sec4lemma_ii}
$S(T_pf,r/2)\subseteq \Omega _p(S(f,r))\subseteq S(T_pf,r)$.
\end{enumerate}
\end{lemma}

\begin{proof}
\ref{sf-sec4lemma_i}
The continuous functions $\widehat{L_sf}$ and $L_{\varepsilon (s)}\widehat{f}$
[$\widehat{T_pf}$ and $\rho _p\widehat{f}$] agree on $\varepsilon (S)$.

\ref{sf-sec4lemma_ii}
If $s\in S(T_pf,r/2)$,
then $\lvert \widehat{L_sf}(p)\rvert \leq r/2$,
and so $\lambda _s^{-1}(S(f,r))=S(L_sf,r)\in p$,
thus proving the first inclusion.
Next, if $s\in \Omega _p(S(f,r))$,
then $S(L_sf,r)\in p$, so $p\in \overline{\varepsilon (S(L_sf,r))}$,
and so $\lvert (T_pf)(s)\rvert \leq r$, thus finishing the proof.
\end{proof}

We apply the following consequence of the Gelfand-Naimark Theorem in
the proofs below:
Let $I$ be an ideal of $\mathcal{F}$, let $\varphi $ be the
$\mathcal{F}$-filter on $S$ generated by $\mathcal{B}(I)$, and let
$f\in \mathcal{F}$.
Then $f\in I$ if and only if
$\widehat{f}(p)=0$ for every $p\in \overline{\varphi }$.

\begin{theorem}
\label{sf-subsemi}
Let $\varphi $ be an $\mathcal{F} $-filter on $S$,
let $I$ be the ideal of $\mathcal{F}$ such that $\varphi $ is
generated by $\mathcal{B}(I)$,
and let $C=\overline{\varphi }$.
The following statements are equivalent:
\begin{enumerate}[\upshape (i)]
\item\label{sf-subsemi_i}
$C$ is a subsemigroup of $\delta S$.
\item\label{sf-subsemi_ii}
If $f\in I$ and $p\in C$, then $T_pf\in I$.
\item\label{sf-subsemi_iii}
If $A\in \varphi $ and $p\in C$,
then $\Omega _p(A)\in \varphi $.
\end{enumerate}
\end{theorem}

\begin{proof}
(\ref{sf-subsemi_i}) $\Rightarrow $ (\ref{sf-subsemi_ii})
Suppose that $C$ is a subsemigroup of $\delta S$.
If $f\in I$ and $p\in C$,
then the function $\widehat{T_pf}$ vanishes on $C$
by Lemma \ref{sf-sec4lemma} \ref{sf-sec4lemma_i}, and so
$T_pf\in I$.

(\ref{sf-subsemi_ii}) $\Rightarrow $ (\ref{sf-subsemi_iii})
Suppose that (ii) holds.
Let $A\in \varphi $ and let $p\in C$.
Pick $f\in I$ and $r>0$ such that $S(f,r)\subseteq A$.
Then $S(T_pf,r/2)\subseteq \Omega _p(A)$ by Lemma
\ref{sf-sec4lemma} \ref{sf-sec4lemma_ii}.
Since $S(T_pf,r/2)\in \varphi $ by assumption,
we have $\Omega _p(A)\in \varphi $, as required.

(\ref{sf-subsemi_iii}) $\Rightarrow$ (\ref{sf-subsemi_i})
Suppose that (\ref{sf-subsemi_iii}) holds.
Let $p,q\in C$.
If $A\in \varphi $, then $A^{\circ }\in \varphi $, so $\Omega
_q(A^{\circ })\in
\varphi $ by assumption, and so $\Omega _q(A^{\circ })\in p$.
Then $pq\in \text{cl}_{\delta S}(\widehat{A^{\circ }})$ by Lemma
\ref{sf-product} (\ref{sf-sfp_iv}), and so $pq\in \overline{\varepsilon (A^{\circ
    })}$
by (\ref{sf:eq:3}).
Therefore, $pq\in \overline{\varphi }=C$, thus finishing the proof.
\end{proof}

An  application of the previous theorem to a single element
$p$ of $\delta S$ implies the next corollary.
The content of the next corollary is exactly the
same as in the case that $S$ is discrete and $\delta S=\beta S$
(see \cite[p. 76]{HindStra}).

\begin{corollary}
An element $p\in \delta S$ is an idempotent if and only if
$\Omega _p(A)\in p$ for every $A\in p$.
\end{corollary}

Next, we proceed to characterize left, right, and two-sided ideals
of $\delta S$.

\begin{definition}
\label{sf-leftthick_def}
An $\mathcal{F}$-filter $\varphi $ on $S$ is \emph{left}
[\emph{right}] \emph{thick} if and only if,
for every set $A\in \varphi $ and for every $s\in S$, there exists a set
$B\in \varphi $ such that  $sB\subseteq A$ [$Bs\subseteq A$].
\end{definition}

Note that an $\mathcal{F}$-filter $\varphi $ on $S$ is left [right] thick if
and only if $s^{-1}A\in \varphi $ [$As^{-1}\in \varphi $] for every
$A\in \varphi $ and for every $s\in S$.

\begin{theorem}
\label{sf-left}
Let $\varphi $ be an $\mathcal{F} $-filter on $S$,
let $I$ be the ideal of $\mathcal{F}$ such that $\varphi $ is
generated by $\mathcal{B}(I)$,
and let $L=\overline{\varphi }$.
The following statements are equivalent:
\begin{enumerate}[\upshape (i)]
\item\label{sf-left_i}
$L$ is a left ideal of $\delta S$.
\item\label{sf-left_ii}
$I$ is left translation invariant.
\item\label{sf-left_iii}
$\varphi $ is left thick.
\end{enumerate}
\end{theorem}

\begin{proof}
(\ref{sf-left_i}) $\Rightarrow $ (\ref{sf-left_ii})
If $f\in I$ and $s\in S$,
then the function $\widehat{L_sf}$ vanishes on $L$.

(\ref{sf-left_ii}) $\Rightarrow $ (\ref{sf-left_iii})
This follows from the equality $S(L_sf,r)=\lambda _s^{-1}(S(f,r))$ and
from the fact that $\mathcal{B}(I)$ is a filter base for $\varphi $.

(\ref{sf-left_iii}) $\Rightarrow $ (\ref{sf-left_i})
Suppose that $\varphi $ is left thick.
To prove the statement,
it is enough to show that $\varepsilon (s)q\in L$ for every $s\in S$ and for
every $q\in L$.
So, let $s\in S$, let $q\in L$, and let $A\in \varphi $.
Pick a set $B\in \varphi $ with $sB\subseteq A$.
Since $q\in \overline{\varepsilon (B)}$,
we have $\varepsilon (s)q\in \overline{\varepsilon (A)}$, as required.
\end{proof}

Similar arguments as given in the proof of Theorem \ref{sf-subsemi}
apply to prove the next theorem, so we leave the details to the reader.

\begin{theorem}
\label{sf-right}
Let $\varphi $ be an $\mathcal{F} $-filter on $S$,
let $I$ be the ideal of $\mathcal{F}$ such that $\varphi $ is
generated by $\mathcal{B}(I)$,
and let $R=\overline{\varphi }$.
The following statements are equivalent:
\begin{enumerate}[\upshape (i)]
\item\label{sf-right_i}
$R$ is a right ideal of $\delta S$.
\item\label{sf-right_ii}
If $f\in I$ and $p\in \delta S$, then
  $T_pf\in I$.
\item\label{sf-right_iii}
If $A\in \varphi $ and $p\in \delta S$,
then $\Omega _p(A)\in \varphi $.
\end{enumerate}
\end{theorem}

Statement (ii) of the previous theorem implies that
$R_sf\in I$ for every $f\in I$ and for every $s\in S$,
and so $I$ is right translation invariant.
Then the equality $S(R_sf,r)=\rho _s^{-1}(S(f,r))$ implies that $\varphi $
is right thick.

The closure $\overline{\varphi }$ of a right thick
$\mathcal{F}$-filter $\varphi $ need not be a right ideal of $\delta S$.
If $S$ is commutative, then an $\mathcal{F}$-filter $\varphi $ on $S$
is left thick if and only if $\varphi $ is right thick.
But, in general, a closed left ideal of $\delta S$
need not be a right ideal of $\delta S$.
For example, let $\mathbb{R}^{\ast }$ denote the two-point
compactification of $\mathbb{R}$, that is $\mathbb{R}^{\ast
}=\mathbb{R} \cup \{ -\infty ,\infty \}$.
Then $\mathbb{R}^{\ast }$ is a semigroup compactification of
$\mathbb{R}$ (with $\infty $ and $-\infty $ as its right zero
elements).
The set $L=\{ \infty \}$ is a left ideal of $\mathbb{R}^{\ast }$.
Since $\infty (-\infty )=-\infty $, the set $L$ is not a right ideal
of $\mathbb{R}^{\ast }$.

Recall that $\delta S$ is a
semitopological semigroup if and only if $\mathcal{F} \subseteq \mathcal{WAP}
(S)$ (see \cite[pp. 138, 143]{BJM}).
The proof of the next theorem follows similar arguments as given in the proof of
Theorem \ref{sf-left}, so
we leave the details to the reader.

\begin{theorem}
\label{sf-rightwap}
Suppose that $\mathcal{F} \subseteq \mathcal{WAP}(S)$.
Let $\varphi $ be an $\mathcal{F} $-filter on $S$,
let $I$ be the ideal of $\mathcal{F}$ such that $\varphi $ is
generated by $\mathcal{B}(I)$,
and let $R=\overline{\varphi }$.
The following statements are equivalent:
\begin{enumerate}[\upshape (i)]
\item\label{sf-rightwap_i}
$R$ is a right ideal of $\delta S$.
\item\label{sf-rightwap_ii}
$I$ is right translation invariant.
\item\label{sf-rightwap_iii}
$\varphi $ is right thick.
\end{enumerate}
\end{theorem}

Combining Theorem \ref{sf-left} and Theorem
\ref{sf-right}, we obtain the following statement.

\begin{theorem}
Let $\varphi $ be an $\mathcal{F} $-filter on $S$,
let $I$ be the ideal of $\mathcal{F}$ such that $\varphi $ is
generated by $\mathcal{B}(I)$,
and let $J=\overline{\varphi }$.
The following statements are equivalent:
\begin{enumerate}[\upshape (i)]
\item $J$ is an ideal of $\delta S$.
\item $I$ is left translation invariant and $T_pf\in I$
for every $f\in I$ and for every $p\in \delta S$.
\item $\varphi $ is left thick and
  $\Omega _p(A)\in \varphi $ for every $A\in \varphi $ and 
  for every $p\in \delta S$.
\end{enumerate}
\end{theorem}

The previous theorem implies that every closed ideal $J$ of
$\delta S$ yields another semigroup compactification of $S$.
Indeed, let $J$ and $I$ be as above and put $\mathcal{F}'=I\oplus
\mathbb{C}$, where $\mathbb{C}$ denotes the constant functions on $S$.
If $J$ is an ideal of $\delta S$, then statement
(\ref{sf-right_ii}) of the previous theorem implies that
$\mathcal{F}'$ is an $m$-admissible subalgebra of $C(S)$.

If $S$ is commutative, then the description of ideals of $\delta S$
is much simpler.

\begin{theorem}
\label{sf-commideal}
Suppose that $S$ is commutative.
Let $\varphi $ be an $\mathcal{F} $-filter on $S$,
let $I$ be the ideal of $\mathcal{F}$ such that $\varphi $ is
generated by $\mathcal{B}(I)$,
and let $J=\overline{\varphi }$.
The following statements are equivalent:
\begin{enumerate}[\upshape (i)]
\item\label{sf-commideal_i}
$J$ is an ideal of $\delta S$.
\item\label{sf-commideal_ii}
If $f\in I$ and $p\in \delta S$, then $T_pf\in I$.
\item\label{sf-commideal_iii}
If $A\in \varphi $ and $p\in \delta S$,
then $\Omega _p(A)\in \varphi $.
\end{enumerate}
\end{theorem}

\begin{proof}
By Theorems \ref{sf-right} and \ref{sf-left},
we need only to show that statement (\ref{sf-commideal_iii}) implies
that $\varphi $ is left thick.
But if (\ref{sf-commideal_iii}) holds, then $J$ is a right ideal of
$\delta S$ by Theorem \ref{sf-right},
so $\varphi $ is right thick (see the paragraph after Theorem
\ref{sf-right}),
and so $\varphi $ is left thick.
\end{proof}

\section{Smallest ideal of $\delta S$ and its closure}
\label{sec:minimal-ideal}

In this last section, we characterize those points of $\delta S$ which are
either in the smallest ideal $K$ of $\delta S$ or in its closure $\overline{K}$.
For the Stone-\v Cech compactification of a discrete semigroup $S$,
a characterization of these points using ultrafilters is given in
\cite[Section 4.4]{HindStra}.
These points were characterized by the author using filters in
\cite{ufilters} in the $\mathcal{LUC}$-compactification of a
topological group.
For the $\mathcal{LMC}$-compactification of a locally compact
semitopological semigroup, a characterization of these points was given
in \cite{toot1} in terms of equivalence classes of $z$-ultrafilters.

The results concerning the points of $K$ and $\overline{K}$
in the previous references are somewhat similar:
The points in $K$ are related to syndetic subsets (defined below) of
$S$, and the points in $\overline{K}$ are related to piecewise
syndetic subsets (defined below) of $S$.
We show that syndetic and piecewise syndetic subsets of
$S$ describe the points of $K$ and $\overline{K}$, respectively, in
any semigroup compactification of $S$ for any semitopological
semigroup $S$.
Therefore,
the results given below are generalizations of the results obtained in
the above references.

Some of the proofs given below follow similar arguments as
given in \cite[Section 4.4]{HindStra} (and in \cite{ufilters}).
However, these proofs usually require a number of small adjustments, so we give
the details for completeness.
We begin with the following simple lemma.

\begin{lemma}
\label{sf-tauopen}
If $A\subseteq S$ is $\tau (\mathcal{F})$-open,
then $s^{-1}A$ is $\tau (\mathcal{F})$-open for every $s\in S$.
\end{lemma}

\begin{proof}
Let $A\subseteq S$ be $\tau (\mathcal{F})$-open and let $s\in S$.
If $t\in s^{-1}A$,
then $\varepsilon (s)\varepsilon (t)\in \widehat{A}$,
and so there exists a $\tau (\mathcal{F})$-open subset $B$ of $S$ such that $B\in
\varepsilon (t)$ and $\lambda _{\varepsilon (s)}(\widehat{B})\subseteq \widehat{A}$.
Then $A\in \varepsilon (s)\varepsilon (u)$ for every $u\in B$, and so $B\subseteq s^{-1}A$.
Since $\varepsilon (t)$ is the filter of all $\tau (\mathcal{F})$-neighborhoods
of $t$, the statement follows.
\end{proof}

In what follows, we apply the fact that
$K$ is the union of all minimal left ideals of
$\delta S$ (see \cite[p. 34]{HindStra}).

\begin{definition}
A subset $A$ of $S$ is \emph{syndetic} if and only if there
exists a finite subset $F$ of $S$ such that $S=\bigcup _{s\in F}s^{-1}A$.
\end{definition}

\begin{theorem}
\label{sf-syndetic}
If $p\in \delta S$, then the following statements are equivalent:
\begin{enumerate}[\upshape (i)]
\item \label{sf-syndetic_i}
$p\in K$.
\item \label{sf-syndetic_ii}
If $A\in p$, then $\Omega _p(A)$ is syndetic.
\item \label{sf-syndetic_iii}
If $q\in \delta S$, then $p\in (\delta S) qp$.
\end{enumerate}
\end{theorem}

\begin{proof}
(\ref{sf-syndetic_i}) $\Rightarrow $ (\ref{sf-syndetic_ii})
Suppose that $p\in K$.
Let $A\in p$ and
let $L$ be the minimal left ideal of $\delta S$ such that $p\in L$.
If $q\in L$, then $L=(\delta S)q=\overline{\varepsilon (S)q}$,
and so $\widehat{A}\cap \varepsilon (S)q \neq \emptyset $.
Pick some element $s\in S$ such that $\varepsilon (s)q\in \widehat{A}$.
Then $s^{-1}A\in q $.
So, for every element $q\in L$, there exists some
$s\in S$ such that $s^{-1}A\in q$.
Therefore, $L\subseteq \bigcup _{s\in F}\widehat{s^{-1}A}$
for some finite subset $F$ of $S$.

Let $t\in S$.
Since $\varepsilon (t)p \in L$, there exists an element $s\in F$ such that
$s^{-1}A\in \varepsilon (t)p$.
Then $t^{-1}(s^{-1}A)=(st)^{-1}A\in p$,
so $st\in \Omega _p(A)$, and so $t\in s^{-1}\Omega _p(A)$, as required.

(\ref{sf-syndetic_ii}) $\Rightarrow $ (\ref{sf-syndetic_iii})
Suppose that (\ref{sf-syndetic_ii}) holds.
Suppose also that there exists some $q\in \delta S$
such that $p\notin (\delta S)qp$.
Since $\delta S$ is a regular topological space,
there exists a $\tau (\mathcal{F})$-open subset $A$ of $S$ such that $A\in p$ and
$\text{cl} _{\delta S}(\widehat{A})\cap (\delta S)qp =\emptyset $.
Put $B=\Omega _p(A)$ and pick a finite subset $F$ of $S$
such that $S=\bigcup _{t\in F}t^{-1}B$.
Pick $t\in F$ such that $q\in \overline{\varepsilon (t^{-1}B)}$.
Then $\varepsilon (t)q\in \overline{\varepsilon (B)}$ by Lemma \ref{sf-product}
(\ref{sf-sfp_iia}),
and so $\varepsilon (t)qp\in \overline{\varepsilon (A)}$ by Lemma \ref{sf-product}
(\ref{sf-sfp_iva}).
But then $\varepsilon (t)qp\in \text{cl}_{\delta S}(\widehat{A})$ by
(\ref{sf:eq:3}), a contradiction.

(\ref{sf-syndetic_iii}) $\Rightarrow $ (\ref{sf-syndetic_i})
Choose any $q\in K$.
\end{proof}

\begin{definition}
A subset $A$ of $S$ is \emph{piecewise syndetic}
if and only if there exists a finite subset $F$ of $S$
such that the family $\{ s^{-1}(\bigcup _{t\in F}t^{-1}A) : s\in S \}$ has the
finite intersection property.
\end{definition}


\begin{theorem}
\label{sf-piece}
The following statements hold for a subset $A$ of $S$:
\begin{enumerate}[\upshape (i)]
\item\label{sf-piece_i}
If $\widehat{A}\cap K\neq \emptyset $,
then $A$ is piecewise syndetic.
\item\label{sf-piece_ii}
If $A$ is $\tau (\mathcal{F})$-open and piecewise syndetic,
then $\text{cl} _{\delta S}(\widehat{A}) \cap K\neq \emptyset $.
\end{enumerate}
\end{theorem}

\begin{proof}
(\ref{sf-piece_i})
Suppose that $p\in \widehat{A}\cap K$.
By Theorem \ref{sf-syndetic}, there exists a finite subset $F$ of $S$
with $S=\bigcup _{t\in F}t^{-1}\Omega _p(A)$.
Let $s\in S$ and pick $t\in F$ such that $ts\in \Omega _p(A)$.
Then $s^{-1}(t^{-1}A)\in p$,
so $s^{-1}(\bigcup _{t\in F}t^{-1}A)\in p$
for every $s\in S$, and so $A$ is piecewise syndetic.

(\ref{sf-piece_ii})
Suppose that $A$ is $\tau (\mathcal{F})$-open and piecewise syndetic.
Pick a finite subset $F$ of $S$ such that the family $\{ s^{-1}(\bigcup
_{t\in F}t^{-1}A) : s\in S\}$ has the finite intersection property.
Put $B=\bigcup _{t\in F}t^{-1}A$.
Then $B$ is a $\tau (\mathcal{F})$-open subset of $S$ by Lemma \ref{sf-tauopen}.
The set $C=\bigcap _{s\in S} \overline{\varepsilon (s^{-1}B)}$ is a non-empty
subset of $\delta S$, so pick some element $p\in C$.

If $s\in S$, then
$p\in \text{cl}_{\delta S}(\widehat{s^{-1}B})$ by Lemma \ref{sf-tauopen} and
(\ref{sf:eq:3}),
and so $\varepsilon (s)p\in \text{cl}_{\delta S}(\widehat{B})$ by
Lemma \ref{sf-product} (\ref{sf-sfp_ii}).
Therefore, $(\delta S)p\subseteq
\text{cl}_{\delta S}(\widehat{B})$.
Here, $(\delta S)p$ is a left ideal of $\delta S$,
and so we may pick some element $q\in K\cap (\delta S)p$.
Since $q\in \text{cl} _{\delta S}(\widehat{B})$,
we have $B\cap C\neq \emptyset $ for every $C\in q$.
By the definition of the set $B$,
we may assume that there exists an element $t\in F$ such that
$t^{-1}A\cap C\neq \emptyset $ for every $C\in q$.
Then
$q\in \text{cl}_{\delta S}(\widehat{t^{-1}A})$ by Lemma \ref{sf-tauopen},
and so $\varepsilon (t)q\in \text{cl}_{\delta S}(\widehat{A})$ by
Lemma \ref{sf-product} (\ref{sf-sfp_ii}).
Since $q\in K$, we have also $\varepsilon (t)q\in K$,
thus finishing the proof.
\end{proof}

\begin{corollary}
If $p\in \delta S$,
then $p\in \overline{K}$ if and only if every $A\in p$ is
piecewise syndetic.
\end{corollary}


\begin{definition}
A subset $A$ of $S$ is \emph{central} if and only if there exists
an idempotent $e\in K$ such that $A\in e$.
\end{definition}

The previous definition depends on the semigroup compactification
$\delta S$ in question, but we hope that the terminology chosen above does
not cause any misunderstandings.

\begin{theorem}
\label{sf:lstthm}
Let $A\subseteq S$.
The implications (\ref{sf:lstthm_i}) $\Rightarrow $
(\ref{sf:lstthm_ii}) $\Rightarrow $ (\ref{sf:lstthm_iii}) $\Rightarrow
$ (\ref{sf:lstthm_v}) hold for the following statements:
\begin{enumerate}[\upshape (i)]
\item \label{sf:lstthm_i}
$\widehat{A}\cap K\neq \emptyset $.
\item \label{sf:lstthm_ii}
The set $\{ s\in S : s^{-1}A \text{ is central} \}$ is syndetic.
\item \label{sf:lstthm_iii}
There exists some $s\in S$ such that $s^{-1}A$ is
  central.
\item \label{sf:lstthm_v}
$A$ is piecewise syndetic.
\end{enumerate}
\end{theorem}

\begin{proof}
(\ref{sf:lstthm_i}) $\Rightarrow $ (\ref{sf:lstthm_ii})
Suppose that $p\in \widehat{A}\cap K$.
Let $L$ be the minimal left ideal of $\delta S$ with $p\in L$ and
pick an idempotent $e\in L$.
Then $A\in p=pe$, so $\Omega _e(A)\in p$ by Lemma
\ref{sf-product} (\ref{sf-sfp_iii}), and so there exists some element
$s\in S$ such that $s^{-1}A\in e$.
Put $B=\Omega _e(s^{-1}A)$.
By Theorem \ref{sf-syndetic}, there exists 
a finite subset $F$ of $S$ such that $S=\bigcup _{t\in
  F}t^{-1}B$.
Put $C=\{ s\in S : s^{-1}A \text{ is central} \}$.
Now, it is enough to show that $S=\bigcup _{u\in sF} u^{-1}C$.
But if $v\in S$, then there exists an element $t\in F$ such that
$tv\in B$,
so $(tv)^{-1}(s^{-1}A)=(stv)^{-1}A\in e$,
and so $stv\in C$.
Therefore, $v\in (st)^{-1}C$, as required.

(\ref{sf:lstthm_ii}) $\Rightarrow $ (\ref{sf:lstthm_iii})
A syndetic subset of $S$ is non-empty.


(\ref{sf:lstthm_iii}) $\Rightarrow $ (\ref{sf:lstthm_v})
Suppose that (\ref{sf:lstthm_iii}) holds.
Pick an element $s\in S$ and an idempotent $e\in K$ with $s^{-1}A\in e$.
By Theorem \ref{sf-syndetic}, there exists a finite subset $F$ of $S$
such that $S=\bigcup _{t\in F} t^{-1}\Omega _e(s^{-1}A)$.
Put $B=\bigcup _{t\in sF} t^{-1}A$.
Now, it is enough to show that $u^{-1}B\in e$ for every $u\in S$.
If $u\in S$, then $tu\in \Omega _e(s^{-1}A)$ for some $t\in F$,
so $(tu)^{-1}(s^{-1}A)=u^{-1}((st)^{-1}A)\in e$,
and so $u^{-1}B\in e$, as required.
\end{proof}

We finish this paper with some examples showing that, in general, the
implications (\ref{sf:lstthm_i}) $\Rightarrow $ (\ref{sf:lstthm_ii})
and (\ref{sf:lstthm_iii}) $\Rightarrow $ (\ref{sf:lstthm_v}) in the
previous theorem can not be reversed.
At the moment, we are unable to determine whether the implication
(\ref{sf:lstthm_ii}) $\Rightarrow $ (\ref{sf:lstthm_iii})
can be reversed.

First, let us consider the multiplicative semigroups $S=[0,1[$ and
$\delta S=[0,1]$ under their Euclidean topologies.
Now, $K=\{0\}$.
Clearly, any subset of $S$ containing the element $0$ is syndetic.
If $A=\{0\}$ and $s=0$, then the set $s^{-1}A=S$ is central.
Since $\widehat{A}=\emptyset $, we see that statement
(\ref{sf:lstthm_ii}) need not imply statement (\ref{sf:lstthm_i}).

To see that the implication (\ref{sf:lstthm_v}) $\Rightarrow $
(\ref{sf:lstthm_iii}) need not hold,
consider the one-point compactification $\delta \mathbb{R}$
of $\mathbb{R}$, that is, $\delta \mathbb{R}=\mathbb{R}\cup \{\infty
\}$.
In this case, $K=\{ \infty \}$.
Put
\[
A= \bigcup _{k=-\infty }^{\infty } [2k-1,2k].
\]
Then $A\cup (-1+A)=\mathbb{R}$, and so $A$ is a piecewise syndetic subset of $\mathbb{R}$.
Let $s\in [0,1[$ and put $s_k=2k-s+1/2$ for every $k\in \mathbb{N}$.
Then the sequence $(s_k)$ converges to $\infty $ in $\delta \mathbb{R}$ and
satisfies $s_k\notin \widehat{-s+A}$ for every $k\in \mathbb{N}$.
Therefore, $-s+A\notin \infty $.
If $t\in \mathbb{R}$ is any element, then $t=s+n$ for some $s\in
[0,1[$ and $n\in \mathbb{Z}$.
The equality $-t+A=-s+A$ implies that $-t+A$ is not central.


\affiliationone{
   T. Alaste\\
   University of Oulu\\
   Department of Mathematical Sciences\\
   P.O. Box 3000\\
   FI-90014 University of Oulu\\
   Finland
   \email{tomi.alaste@gmail.com}}
\end{document}